\documentclass[preprint,12pt]{elsarticle}

\usepackage{amscd}
\usepackage{amsmath}
\usepackage{amsthm}
\usepackage{amsfonts}
\usepackage{graphicx}
\usepackage{amssymb}
\usepackage{color}
\usepackage{amsbsy}
\usepackage{graphicx}
\usepackage{amssymb,color,amsbsy}%, amsthm}
\usepackage{float}
\usepackage{comment}

\usepackage[ruled]{algorithm2e}

\usepackage[matrix,arrow]{xy}
\usepackage{mathabx}   
\DeclareMathAlphabet{\mathpzc}{OT1}{pzc}{m}{it}

\usepackage{pgfpages}
\usepackage{tikz}
\usetikzlibrary{shapes.geometric}
\usetikzlibrary{decorations.pathmorphing}
\usetikzlibrary{arrows}
\usetikzlibrary{backgrounds}
\usetikzlibrary{positioning}
\usetikzlibrary{fit}
\usepackage{caption}
\usepackage{amsthm}

\usepackage{amsmath}

\usepackage[pagebackref=true, bookmarksopen=true,colorlinks=true, linkcolor=red,citecolor=blue]{hyperref}

\usepackage[capitalise]{cleveref}  %References will output ‘Theorem 3.21’, without you having to enter the “Theorem.”

\global\long\def\ii{\cap}%

\global\long\def\u{\cup}%

\global\long\def\I{\bigcap}%

\global\long\def\U{\bigcup}%

\global\long\def\s{\subset}%

\global\long\def\S{\subset}%

\global\long\def\p{\prime}%

\global\long\def\P{\prime}%

\global\long\def\sm{\sum}%

\global\long\def\ñ{\sim}%

\global\long\def\Le{\le}%

\usepackage{tikz-cd}

\usepackage{tkz-graph}
\usepackage{multicol}

\usepackage{subcaption}

\newtheorem{theorem}{Theorem}[section]

\newtheorem{conjecture}[theorem]{Conjecture}
\newtheorem{corollary}[theorem]{Corollary}

\newtheorem{lemma}[theorem]{Lemma}

\newtheorem{problem}[theorem]{Problem}
\newtheorem{proposition}[theorem]{Proposition}

\newtheorem{observation}[theorem]{Observation}
\usepackage{comment}
\begin{document}

\begin{abstract}
	
An almost bipartite graph is a graph with a unique odd cycle. Levit
and Mandrescu showed that in every non-König--Egerváry almost bipartite
graph the equalities $\textnormal{ker}(G)=\textnormal{core}(G)$, 
$\textnormal{corona}(G)\cup N(\textnormal{core}(G)) = V(G)$
and $\left|\textnormal{corona}(G)\right|+\left|\textnormal{core}(G)\right|=2\alpha(G)+1$
hold. In this work, we present a generalization of this theory by
introducing the family of $R$-disjoint graphs, which contains all
non-König--Egerváry almost bipartite graphs, allowing the presence of
multiple odd cycles under connectivity constraints based on
the reach sets $R(C)$. We prove that $R$-disjoint graphs
preserve the fundamental properties of almost bipartite graphs:
$\textnormal{ker}(G)=\textnormal{core}(G)$ and
$\textnormal{corona}(G)\cup N(\textnormal{core}(G))=V(G)$.
Moreover, we establish the formula 
$\left|\textnormal{corona}(G)\right|+\left|\textnormal{core}(G)\right|=2\alpha(G)+k$,
where $k$ is the number of disjoint odd cycles in $G$, 
which refines the previously known particular case when $k=1$. 
$R$-disjoint graphs naturally induce a canonical decomposition; 
we obtain structural properties of this decomposition and, 
as a consequence, verify a recent conjecture of Levit and Mandrescu.

\end{abstract}

\begin{frontmatter} 
	\title{On R-disjoint graphs: a generalization of almost bipartite non–König–Egerváry graphs}

	\author[pan,daj]{Kevin Pereyra} 	\ead{kdpereyra@unsl.edu.ar}

	\address[pan]{Universidad Nacional de San Luis, Argentina.} 	\address[daj]{IMASL-CONICET, Argentina.} 
	\date{Received: date / Accepted: date} 
	
\end{frontmatter} %

\section{Introduction}
If $\alpha(G)+\mu(G)$ equals the order of the graph $G$, then
$G$ is a König--Egerváry graph \cite{deming1979independence,gavril1977testing,stersoul1979characterization}.
Various properties of König--Egerváry graphs are presented in 
\cite{bourjolly2009node,jarden2017two,levit2006alpha,levit2012critical}.
It is known that every bipartite graph is a König--Egerváry graph
\cite{egervary1931combinatorial}.  

We define the sets $\Omega^{*}(G)=\left\{ S:S\text{ is an independent set of }G\right\}$,
$\Omega(G)=\left\{ S: \left|S\right|=\alpha(G) \right\}$, 
$\textnormal{core}(G)=\bigcap\left\{ S:S\in\Omega(G)\right\}$ \cite{levit2003alpha+},
and $\textnormal{corona}(G)=\bigcup\{ S:$ $S\in\Omega(G)\}$ \cite{boros2002number}.
The number $d_{G}(X)=\left|X\right|-\left|N(X)\right|$ is called the difference
of the set $X\s V(G)$, and $d(G)=\max\{d_{G}(X):X\s V(G)\}$ is called
the critical difference of $G$. A set $U\s V(G)$ is critical
if $d_{G}(U)=d(G)$ \cite{zhang1990finding}. The number 
$d_{I}(G)=\max\left\{ d_{G}(X):X\in\Omega^{*}(G)\right\}$ 
is called the critical independence difference of $G$. 
If a set $X\s \Omega^{*}(G)$ satisfies $d_{G}(X)=d_{I}(G)$, then it is called
a critical independent set \cite{zhang1990finding}. Clearly, $d(G)\ge d_{I}(G)$
holds for every graph. It is known that $d(G)=d_{I}(G)$ for every
graph \cite{zhang1990finding}. We define 
$\textnormal{ker}(G)=\bigcap\left\{ S:S\text{ is a critical independent set of }G\right\}$
\cite{levit2012vertices,lorentzen1966notes,schrijver2003combinatorial}.

It is known that $\textnormal{ker}(G)\s \textnormal{core}(G)$ for every graph 
\cite{levit2012vertices}, while equality holds in bipartite graphs 
\cite{levit2013critical}, and for unicyclic non-König--Egerváry graphs 
\cite{levit2011core}.  

A graph $G$ is almost bipartite if it has only one odd cycle. In
\cite{levit2025almost} it is shown that the equality between $\ker(G)$
and $\textnormal{core}(G)$ also holds for every almost bipartite 
non-König--Egerváry graph. That is, we have the following:

\begin{theorem}  [\cite{levit2025almost}\label{pp}]  
	If $G$ is an almost bipartite non-König--Egerváry graph, then 
	\[
	\ker(G)=\textnormal{core}(G).
	\]
\end{theorem}

The previous result motivates our work. In \cref{smat}, we
define the family of $R$-disjoint graphs, which is a family
containing the almost bipartite non-König--Egerváry graphs. 
This family allows the existence of multiple odd cycles in a graph,
but restricts the way these cycles can be connected. We prove
that for any $R$-disjoint graph the equality $\ker(G)=\textnormal{core}(G)$ holds.

For every $R$-disjoint graph, we show that 
$\textnormal{corona}(G)\cup N(\textnormal{core}(G))=V(G)$,
while 
$\left|\textnormal{corona}(G)\right|+\left|\textnormal{core}(G)\right|=2\alpha(G)+k$,
where $k$ is the number of disjoint odd cycles in $G$. 
This refines some previously known results.  

\begin{theorem}
	[\label{levit}\cite{levit2025almost}] If $G$ is an almost bipartite
	non-König--Egerváry graph, then 
	\begin{itemize}
		\item $\ker(G)=\textnormal{core}(G),$
		\item $\left|\textnormal{corona}(G)\right|+\left|\textnormal{core}(G)\right|=2\alpha(G)+1,$
		\item $\textnormal{corona}(G)\cup N(\textnormal{core}(G))=V(G).$
	\end{itemize}
\end{theorem}

The paper is organized as follows: in \cref{spre} we present preliminary
results and notation used throughout this work. In \cref{smat}, 
we define $R$-disjoint graphs and study the structure of matchings 
in these graphs. In \cref{sind}, we study independent sets in this
class of graphs and prove the main results of the paper. These results
provide a partial answer to Problem 3.1 posed in \cite{levit2025almost}, 
extending the validity of the equality $\ker(G)=\textnormal{core}(G)$ 
to a family of graphs with multiple odd cycles. In \cref{scon}, 
we prove structural results in $R$-disjoint graphs, 
including a result motivated by a conjecture of Levit and Mandrescu 
in \cite{levit2025almost}. As a particular case of this result 
(applied to an almost bipartite non-König--Egerváry graph), 
the conjecture is verified.

\section{Preliminaries}\label{spre}
All graphs considered in this paper are finite, undirected, and simple.
For any undefined terminology or notation, we refer the reader to
Lovász and Plummer \cite{LP} or Diestel \cite{Distel}.

Let \( G = (V, E) \) be a simple graph, where \( V = V(G) \) is the finite set of vertices and \( E = E(G) \) is the set of edges, with \( E \subseteq \{\{u, v\} : u, v \in V, u \neq v\} \). We denote the edge \( e=\{u, v\} \) as \( uv \). A subgraph of \( G \) is a graph \( H \) such that \( V(H) \subseteq V(G) \) and \( E(H) \subseteq E(G) \). A subgraph \( H \) of \( G \) is called a \textit{spanning} subgraph if \( V(H) = V(G) \). 

Let \( e \in E(G) \) and \( v \in V(G) \). We define \( G - e := (V, E \setminus \{e\}) \) and \( G - v := (V \setminus \{v\}, \{uw \in E : u,w \neq v\}) \). If \( X \subseteq V(G) \), the \textit{induced} subgraph of \( G \) by \( X \) is the subgraph \( G[X]=(X,F) \), where \( F:=\{uv \in E(G) : u, v \in X\} \).

Given a vertex set $S \subseteq V(G)$, we denote by $\partial(S)$
the set of edges with one endpoint in $S$ and the other in $V(G)\setminus S$.
We also denote by $\partial(S)$ the set of vertices in $S$
that are incident to an edge with the other endpoint in \textbf{$V(G)\setminus S$}.

A \textit{matching} \(M\) in a graph \(G\) is a set of pairwise non-adjacent edges. The \textit{matching number} of \(G\), denoted by  \(\mu(G)\), is the maximum cardinality of any matching in \(G\). Matchings induce an involution on the vertex set of the graph: \(M:V(G)\rightarrow V(G)\), where \(M(v)=u\) if \(uv \in M\), and \(M(v)=v\) otherwise. If \(S, U \subseteq V(G)\) with \(S \cap U = \emptyset\), we say that \(M\) is a matching from \(S\) to \(U\) if \(M(S) \subseteq U\).

A matching $M$ is \emph{perfect} if $M(v)\neq v$ for every vertex
of the graph. A graph $G$ is \textit{factor-critical} if $G-v$ has
a perfect matching for every vertex $v$. 

A vertex set \( S \subseteq V \) is \textit{independent} if, for every pair of vertices \( u, v \in S \), we have \( uv \notin E \). The number of vertices in a maximum independent set is denoted by \( \alpha(G) \). By definition, the empty set is independent. A bipartite graph is a graph whose vertex set can be partitioned into two disjoint independent sets. The number of vertices in a graph is called the order of the graph.

A vertex set $C \subseteq V(G)$ is a \textit{vertex cover} if every edge of $G$ has at least one endpoint in $C$. A vertex cover of minimum cardinality is called a \textit{minimum vertex cover}.

A \textit{cycle} in $G$ is called \textit{odd} (\textit{even}) if it has an odd (even) number of edges. An \textit{even alternating cycle} with respect to a matching $M$ is an even cycle whose edges alternate between belonging and not belonging to $M$. 

Let $M$ be a matching in $G$. A path (or walk) is called \textit{alternating} with respect to $M$ if, for each pair of consecutive edges in the path, exactly one of them belongs to $M$. If the matching is clear from context, we simply say that the path is alternating. Given an alternating path (or walk), we say that $P$ is: \textit{\(mm\)-\(M\)path} if it starts and ends with edges in $M$, \textit{\(nn\)-\(M\)path} if it starts and ends with edges not in $M$, \textit{\(mn\)-\(M\)path} if it starts with an edge in $M$ and ends with one not in $M$, and \textit{\(nm\)-\(M\)path} if it starts with an unmatched edge and ends with a matched edge. 

For vertices $u,v \in V(G)$, we denote by $d_G(u,v)$ the length of a shortest path between $u$ and $v$ in $G$.

\section{Matching structure of R-disjoint graphs}\label{smat}
We begin this section by mentioning some known results in the structural
theory of non-König--Egerváry graphs.

$ $

In \cite{edmonds1965paths}, Edmonds introduced the following concepts
relative to a matching $M$ of a graph $G$ and its subgraphs. An
$M$-blossom of $G$ is an odd cycle of length $2k+1$ with $k$ edges
in $M$. The vertex not saturated by $M$ in the cycle is called
the \textit{base} of the blossom. An $M$-stem is an $M$-alternating
path of even length (possibly zero) connecting the base of the blossom
with a vertex not saturated by $M$. The base is the only common vertex
between the blossom and the stem. An $M$-flower is a blossom joined
with a stem. The vertex not saturated by $M$ in the stem is called
the \textit{root} of the flower.

In \cite{stersoul1979characterization}, Sterboul introduced the concept
of a \textit{posy} for the first time. An $M$-posy consists of two
vertex-disjoint blossoms joined by an $mm$-$M$path. The endpoints of the
path are exactly the bases of the two blossoms. There are no internal
vertices of the path in the blossoms.

\begin{theorem}\label{safe}
	For a graph $G$, the following properties are equivalent: 
	\begin{itemize}
		\item $G$ is a non-König--Egerváry graph.
		\item For every maximum matching $M$, there exists an $M$-flower or an $M$-posy
		in $G$.
		\item For some maximum matching $M$, there exists an $M$-flower or an $M$-posy
		in $G$.
	\end{itemize}
\end{theorem}

Sterboul \cite{stersoul1979characterization} was the first to characterize
König--Egerváry graphs via forbidden configurations relative to
a maximum matching. Subsequently, Korach, Nguyen,
and Peis \cite{korach2006subgraph} reformulated this characterization
in terms of simpler configurations, unifying the structures of
flowers and posies. Later, Bonomo et al. \cite{bonomo2013forbidden}
obtained a purely structural characterization based on forbidden subgraphs.
More recently, in \cite{jaume2025confpart1,jaume2025confpart2,jaume2025confpart3},
results were obtained that simplify working with flower and posy
structures.

$ $

Given a graph $G$ and an odd cycle $C$ of $G$, we define the \emph{reach
	set} of $C$ as
\[
R(C):=\bigcup_{F}V(F),
\]
where the union is over all $M$-flowers $F$ of $G$
having $C$ as their $M$-blossom, for some maximum matching $M$.

A graph $G$ is \emph{odd cycle disjoint} if every pair of distinct
odd cycles $C$ and $C'$ of $G$ satisfies $V(C)\cap V(C')=\emptyset$.
An $R$-\emph{disjoint graph} $G$ is a graph with at least
one odd cycle, such that $R(C)\neq\emptyset$ and $R(C)\cap R(C')=\emptyset$
for each pair of distinct odd cycles $C$ and $C'$ of $G$.
Therefore, every $R$-\emph{disjoint graph} is also an odd cycle
disjoint graph, so each odd cycle is chordless, and moreover $V(C)\subseteq R(C)$.

In an $R$-disjoint graph $G$, one can naturally consider the
following partition of the vertex set:
\[
\left\{ R(C) : C \text{ is an odd cycle of } G \right\} \cup \{B(G)\},
\]
where $B(G)=V(G)\setminus \bigcup_{C} R(C)$. This is the \emph{flower
	decomposition} of $G$. It is immediate that $G[B(G)]$ is bipartite,
since it contains no odd cycles. Moreover, the only odd cycle in
$G[R(C)]$ is $C$. Therefore, we immediately obtain the following
observation.

$ $

\begin{figure}[h!]
	%Tikz
	
	\begin{center}

		\tikzset{every picture/.style={line width=0.75pt}} %set default line width to 0.75pt        
		


	\end{center}
	\caption{In this figure, we show a graph $G$ of order $18$. A maximum matching is
		highlighted in red, and a maximum independent set in blue. Notice that 
		$\alpha(G)=9$, while $\textnormal{corona}(G)=V(G)\setminus\{4,16\}$, 
		and $\textnormal{ker}(G)=\textnormal{core}(G)=\{1,2,17,18\}$. 
		Hence, $\textnormal{corona}(G)\cup N(\textnormal{core}(G))=V(G)$, 
		and 
		$\left|\textnormal{corona}(G)\right|+\left|\textnormal{core}(G)\right|=16+4=20=2\alpha(G)+2=2\cdot 9+2.$
		Moreover, $G$ contains only $2$ odd cycles.
	}
	\label{Figurapane}
	
\end{figure}
\begin{observation}\label{observ}
	If $G$ is an $R$-disjoint graph and $C$ is
	an odd cycle of $G$, then the following statements hold:
	\begin{itemize}
		\item $G[R(C)]$ is an almost bipartite graph.
		\item $G[B(G)]$ is a bipartite graph.
	\end{itemize}
\end{observation}

By \cref{safe}, every $R$-disjoint graph is a non-König--Egerváry
graph. Hence, a graph $G$ is an $R$-disjoint graph with a unique
odd cycle if and only if $G$ is an almost bipartite non-König--Egerváry
graph. Therefore, $R$-disjoint graphs generalize these graphs,
allowing the existence of multiple odd cycles. In addition to \cref{observ},
we will prove in \cref{obsthm} that each component $G[R(C)]$ is also
a non-König--Egerváry graph.

\begin{theorem}
	[\cite{edmonds1965paths,gallai1964maximale}\label{ge}] Given
	a graph $G$, let 
	\begin{align*}
		D(G) & :=\{ v : \text{there exists a maximum matching missing } v \} \\
		A(G) & :=\{ v : v \text{ is a neighbor of some } u \in D(G), \text{ but } v \notin D(G) \} \\
		C(G) & := V(G) \setminus (D(G) \cup A(G)).
	\end{align*}
	\noindent If $G_1,\dots,G_k$ are the components of $G[D(G)]$
	and $M$ is a maximum matching of $G$, then
	\begin{enumerate}
		\item $M$ covers $C(G)$ and matches $A(G)$ into distinct components of $G[D(G)]$.
		\item Each $G_i$ is a factor-critical graph, and $M$ restricted to $G_i$
		is a near-perfect matching. 
	\end{enumerate}
\end{theorem}

\begin{lemma}\label{lematrivial}
	Every factor-critical graph contains an odd cycle.
\end{lemma}

\begin{proof}
	Notice that a factor-critical graph cannot be bipartite, since
	removing a vertex from the smaller side would result in a graph
	that has no perfect matching.
\end{proof}

An ear decomposition $G_0,G_1,\dots,G_k=G$ of a graph $G$
is a sequence of graphs with the first graph being simple (e.g., a
vertex, edge, even cycle, or odd cycle), and each graph $G_{i+1}$ 
is obtained from $G_i$ by adding an ear. Adding an ear is done
as follows: take two vertices $u$ and $v$ of $G_i$ and add a
path $P_i$ from $u$ to $v$ such that all vertices on the path
except $u$ and $v$ are new vertices.

\begin{theorem}
	[\cite{lovasz1972structure}\label{lovasz}] A graph $G$ is factor-critical
	if and only if $G$ has an odd-length ear decomposition starting from
	an odd cycle. 
\end{theorem}

\begin{corollary}\label{lovaszk}
	If $G$ is an odd cycle disjoint graph, then
	each component of $G[D(G)]$ is either an isolated vertex or an odd cycle.
\end{corollary}

\begin{proof}
	Let $H$ be a connected component of $G[D(G)]$. By
	\cref{ge}, $H$ is a factor-critical graph. Suppose
	that $H$ is neither an odd cycle nor an isolated vertex. Then, by
	\cref{lovasz}, $H$ has an odd-length ear decomposition starting
	from an odd cycle, $H_0,\dots,H_k=H$, with $k\ge 1$. But $H_1$ is
	an odd cycle to which an odd ear is added; it is easy to see that
	$H_1$ would then contain two odd cycles sharing vertices, which
	are also cycles of $G$, a contradiction.
\end{proof}

\begin{lemma}\label{reach}
	Let $G$ be a graph and $M$ a maximum matching. Then,
	for each $x \in D(G)$, either $M(x)=x$, or there exists an $mn$-$M$ path
	from $x$ to a vertex not saturated by $M$ in $G$.
\end{lemma} 

\begin{proof}
	Let $M$ and $M'$ be maximum matchings of $G$ such that
	$M(x)\neq x$ and $M'(x)=x$. Considering the symmetric difference
	$M \triangle M'$ gives a graph composed of alternating
	$MM'$ cycles and even-length alternating $MM'$ paths. Then $x$
	is in one of these alternating paths, and therefore there exists
	an $mn$-$M$ path from $x$ to a vertex not saturated by $M$ in $G$.
\end{proof}

\begin{lemma}\label{as2}
	Let $G$ be an $R$-disjoint graph, $C$ an odd cycle
	of $G$, and $x \in R(C)$. Let $F$ be an $M$-flower of $G$ with base $c \in V(C)$
	containing $x$. Then the following statements hold:
	\begin{itemize}
		\item If $x \in V(C)$, then $x \in D(G)$.
		\item If $x \notin V(C)$ and $d_F(x,c)$ is even, then $x \in D(G)$.
		\item If $x \notin V(C)$ and $d_F(x,c)$ is odd, then $x \in A(G)$.
	\end{itemize}
\end{lemma}

\begin{proof}
	The first two items are immediate, as it is easy
	to modify $M$ to obtain a matching that does not saturate $x$, see
	\cref{Figura1}.

	$ $
	\begin{figure}[h!]
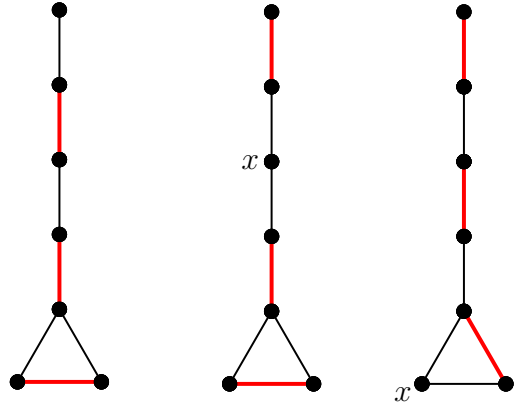

		%Tikz
		
		\begin{center}

			\tikzset{every picture/.style={line width=0.75pt}} %set default line width to 0.75pt        
			


		\end{center}
		\caption{Illustration of Case 1 in the proof of \cref{as2}
		}
		\label{Figura1}
	\end{figure}

	$ $
	Suppose that $x \notin V(C)$ and $d_F(x,c)$ is odd. By contradiction,
	also assume that $x \notin A(G)$, but since $N(x) \cap D(G) \neq \emptyset$,
	it follows that $x \in D(G)$. Let $D_C$ be the component in $G[D(G)]$
	containing the vertices in $V(C)$, and let $D_x$ be the component
	of $G[D(G)]$ containing $x$. Note that $D_x$ has at least
	three vertices.
	
	$ $
	
	\textbf{Case 1.} Suppose that $D_x = D_C$. Then, since $x \in V(D_C) \setminus V(C)$,
	$D_C$ is nontrivial and is not a cycle. This contradicts
	\cref{lovaszk}, since $G$ is odd cycle disjoint.
	
	\textbf{Case 2.} Suppose that $D_x \neq D_C$. Since $D_x$
	has at least three vertices, by \cref{lematrivial} there exists an odd cycle $C'$
	of $G$ containing $x \in V(D_x) \cap R(C')$. Since $x \notin V(C)$,
	we have $C \neq C'$. But $x \in R(C) \cap R(C')$, which
	is a contradiction, as $G$ is an $R$-disjoint graph.
	
\end{proof}
\begin{corollary}\label{1s2}
	Let $G$ be an $R$-disjoint graph, $C$ an odd cycle
	of $G$, and $x \in R(C) - V(C)$. Then 
	\[
	\left|d_F(x,c)\right| \equiv_2 \left|d_{F'}(x,c')\right|.
	\]
	\noindent Where $F$ and $F'$ are an $M$-flower and an $M'$-flower
	of $G$ with bases $c,c' \in V(C)$ respectively such that $x \in V(F) \cap V(F')$,
	for any pair of maximum matchings $M$ and $M'$ of $G$. 
\end{corollary}

According to \cref{1s2}, under the same hypotheses we can
define the following sets: 
\begin{align*}
	R_{\text{odd}}(C) & := \left\{ x \in R(C) - V(C) : \left|d_F(x,c)\right| \equiv_2 1 \right\}, \\
	R_{\text{even}}(C) & := \left\{ x \in R(C) - V(C) : \left|d_F(x,c)\right| \equiv_2 0 \right\}.
\end{align*}
\noindent By \cref{1s2}, we have $R_{\text{odd}}(C) \cap R_{\text{even}}(C) = \emptyset$.
By \cref{as2}, $V(C) \cup R_{\text{even}}(C) \subseteq D(G)$, and $R_{\text{odd}}(C) \subseteq A(G)$.

\begin{lemma}
	Let $G$ be an $R$-disjoint graph, $C$ an odd cycle
	of $G$, and $x \in V(C) \cup R_{\text{even}}(C)$. Then $y \in R(C)$
	for every edge $xy \in E(G)$. 
\end{lemma}

\begin{proof}
	By contradiction, suppose there exists an edge $xy \in E(G)$
	such that $y \notin R(C)$. Then there exists a maximum matching $M$ of
	$G$ and an $M$-flower $F$ of $G$ with base $c \in C$. Without loss
	of generality, we may assume that $x$ is the root of the flower, see
	\cref{Figura1}. Then $y$ is saturated by $M$, otherwise $x,y$ would
	form an augmenting path. We can then modify
	$M$ as follows:
	
	\[
	M' = (M - \{ yM(y) \}) \cup \{ xy \}.
	\]
	
	\noindent It is then easy to see that $G$ has an $M'$-flower
	with root $y$ and $M'$-blossom $C$. Hence $y \in R(C)$, which
	contradicts the initial assumption.
\end{proof}

\begin{lemma}\label{tas2}
	Let $G$ be an $R$-disjoint graph, $C$ an odd cycle
	of $G$, $x \in R(C)$, and $M$ a maximum matching of $G$. Then
	$y \in R(C)$ for every edge $xy \in M$. Furthermore,
	\begin{enumerate}
		\item If a vertex $v \in \partial(B(G))$, then $v \in A(G) \cup C(G),$
		\item If a vertex $v \in \partial(R(C))$, then $v \in A(G).$
	\end{enumerate}
\end{lemma}

\begin{proof}
	By contradiction, suppose there exists an edge $xy \in M$
	such that $y \notin R(C)$. By \cref{as2}, $x \in R_{\text{odd}}(C)$,
	thus $x \in A(G)$. Therefore, by \cref{ge}, $y \in D(G)$.
	If $y \in R(C')$ for some odd cycle $C'$ of $G$ distinct
	from $C$, then by \cref{as2}, $y \in A(G)$, which is impossible,
	so $y \in B(G)$. To finish, we show that $y \in D(G)$ is impossible without relying on $xy \in M$;
	this completes the proof.
	
	Since $x \in R_{\text{odd}}(C)$, there exists a maximum matching $M$ of
	$G$ such that $x$ is in an $M$-flower $F$ with $M$-blossom $C = (c_0 = c), c_1, \dots, (c_k = c)$
	and base $c \in V(C)$. Let $P = (p_0 = c), p_1, \dots, p_n$ be the $M$-stem
	of the flower. Then $x = p_i$ for some odd $i \in \{1, \dots, n\}$.
	On the other hand, $y$ is saturated by $M$, since otherwise $y \in R(C)$,
	as $C$ together with the path $p_0, \dots, p_i, y$ would form an $M$-flower
	with $M$-blossom $C$ containing $y$. Hence, since $y \in D(G)$,
	by \cref{reach}, there exists an $mn$-$M$ path $Q = q_0, \dots, q_t$ starting at $y$ and ending at a vertex unsaturated by $M$. 
	
	If $V(Q) \cap (V(C) \cup V(P)) = \emptyset$, then $Q$ together
	with $C$ and the path $p_0, \dots, p_i, y$ is an $M$-flower with
	$M$-blossom $C$ containing $y$, a contradiction. Let $q_j \in V(Q)$
	be the first vertex of $Q$ in $V(C) \cup V(P)$. Since $Q$ is an $mn$-$M$ path,
	$j$ must be even. Consider the following cases:
	
	$ $
	
	\textbf{Case 1.} If $q_j = p_i$, then $C' = q_0, \dots, q_j, q_0$
	is an odd cycle containing $p_i$, so $R(C) \cap R(C') \neq \emptyset$,
	where $C$ and $C'$ are distinct cycles of $G$, a contradiction (see
	\cref{Figura2}).

	\begin{figure}[h!]
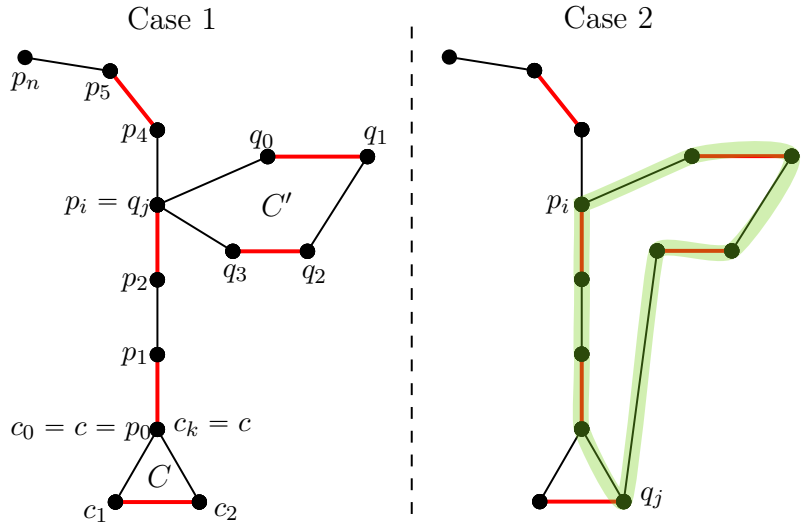

		%Tikz
		
		\begin{center}

			\tikzset{every picture/.style={line width=0.75pt}} %set default line width to 0.75pt        
			


		\end{center}
		\caption{Illustration of the proof of Cases 1 and 2 in \cref{tas2}.
		}
		\label{Figura2}
	\end{figure}
	
\textbf{Case 2.} If $q_{j}\in V(C)$, consider a path of even length between $p_{i}$ and $q_{j}$ contained in $G[V(C)\cup V(P)]$. This path together with the path $p_{i},q_{0},\dots,q_{j}$ forms an odd cycle that intersects $C$, a contradiction (see \cref{Figura2}).

$ $

Now suppose that $q_{j}\in V(P)-V(C)$, that is, $q_{j}=p_{k}$ for some $k\in\{1,\dots,n\}$.

$ $

\textbf{Case 3.} If $k$ is odd, similarly to Case 2, it is easy to obtain an odd cycle containing $p_{k}$ distinct from $C$, a contradiction (see \cref{Figura3}).

	\begin{figure}[h!]
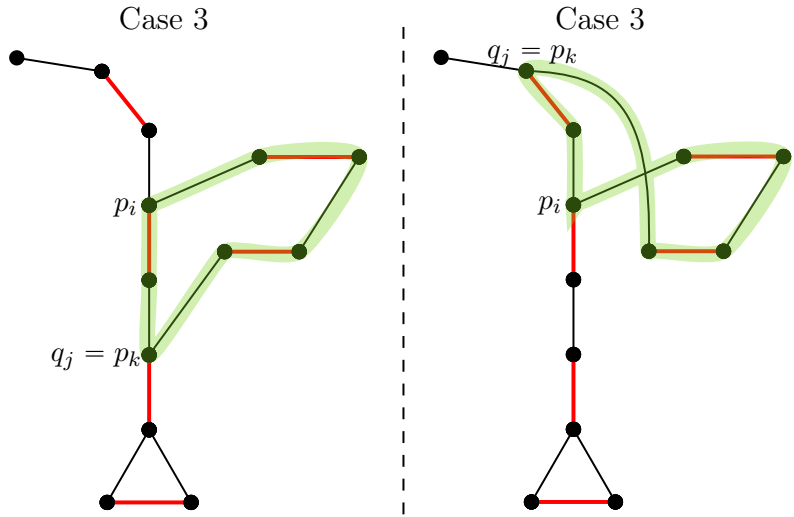

		%Tikz
		
		\begin{center}

			\tikzset{every picture/.style={line width=0.75pt}} %set default line width to 0.75pt        
			


		\end{center}
		\caption{Illustration of the proof of Case 3 in \cref{tas2}.
		}
		\label{Figura3}
	\end{figure}
\textbf{Case 4.} If $k$ is even and $k<j$, we can construct an $M$-alternating cycle such that, by rotating the matching $M$ along the cycle, we obtain a matching $M^{\prime}$ and an $M^{\prime}$-flower with $M^{\prime}$-blossom $C$ containing $q_{0}$, which is absurd (see \cref{Figura4}).

	\begin{figure}[h!]
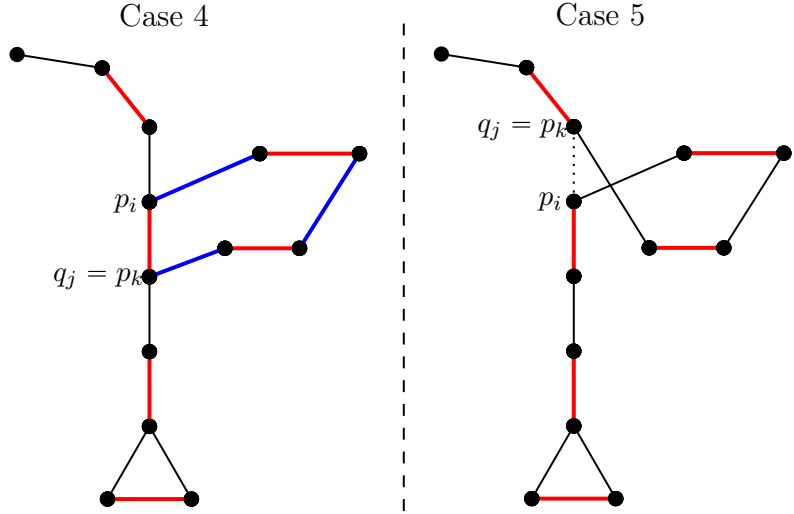

		%Tikz
		
		\begin{center}

			\tikzset{every picture/.style={line width=0.75pt}} %set default line width to 0.75pt        
			
			%

		\end{center}
		\caption{Illustration of the proof of Cases 4 and 5 of \cref{tas2}.
		}
		\label{Figura4}
	\end{figure}
	
	\textbf{Case 5.} If $k$ is even and $k > j$, there exists an $M$-flower with $M$-blossom $C$ containing $q_0$, which is absurd (see \cref{Figura4}).
	
\end{proof}

In other words, the edges connecting two distinct sets
in the flower decomposition of an $R$-disjoint graph never belong
to a maximum matching of the graph. As an immediate consequence, we obtain
the $\mu$-additivity of this decomposition.

\begin{theorem}
	Let $G$ be an $R$-disjoint graph. Then $M$ is
	a maximum matching of $G$ if and only if the restriction of $M$ to each
	set in the flower decomposition of $G$ is a maximum matching
	in the corresponding induced subgraph.
\end{theorem}

\begin{theorem}
	Let $G$ be an $R$-disjoint graph. Then
	\[
	\mu\left(G\right)=\mu\left(B(G)\right)+\underset{C}{\sm}\mu\left(G[R(C)]\right),
	\]
	\noindent where the sum is over all odd cycles $C$ of
	the graph $G$. 
\end{theorem}

Then, by \cref{safe}, \cref{observ} takes the following form. Since an $M$-flower of $G$ with $M$-blossom $C$ is also
an $M^{\P}$-flower of $G[R(C)]$ with $M^{\P}$-blossom $C$, where
$M^{\P}=M\ii E(G[R(C)])$.

\begin{lemma}\label{obsthm}
	If $G$ is an $R$-disjoint graph and $C$ is an odd
	cycle of $G$, then the following statements hold:
	\begin{itemize}
		\item $G[R(C)]$ is an almost bipartite non-König--Egerváry graph.
		\item $G[B(G)]$ is a bipartite graph.
		\item $G[R(C)\u B(G)]$ is an almost bipartite non-König--Egerváry
		graph.
	\end{itemize}
\end{lemma}

\begin{corollary}\label{obscor}
	If $G$ is an $R$-disjoint graph, then $D\left(G\right)=D\left(B(G)\right)\u\underset{C}{\U}D\left(G[R(C)]\right),$
	$A\left(G\right)=A\left(B(G)\right)\u\underset{C}{\U}A\left(G[R(C)]\right),$
	and $C\left(G\right)=C\left(B(G)\right)$, where the union is over all
	odd cycles $C$ of the graph $G$. 
\end{corollary}

We can then rewrite the sets as
\begin{align*}
	R_{\text{odd}}(C) & =R(C)\ii A(G),\\
	R_{\text{even}}(C)\u V(C) & =R(C)\ii D(G).
\end{align*}

\section{Independent and critical sets in R-disjoint graphs}\label{sind}

The following Lemma is a straightforward application of Hall's Theorem.

\begin{lemma}
	[\cite{larson2007note}] Let $G$ be a graph and $I$ a
	critical independent set of $G$. Then there exists a maximum matching
	of $G$ that matches the vertices $N(S)$ into (a subset of) the
	vertices of $I$.
\end{lemma} 

\begin{lemma}\label{indD}
	For any graph $G$, there exists a critical independent
	set $I$ of $G$ (possibly empty) such that $I\s D(G)$. 
\end{lemma}

\begin{proof}
	Let $I$ be a critical independent set of $G$.
	If there exists a set $\emptyset\neq S\s N(I)$ such that $\left|N_{I}(S)\right|=\left|S\right|$,
	then we define $I^{\P}=I-N_{I}(S)$. Now $I^{\P}$ is a critical
	independent set of $G$, since 
	
	\begin{align*}
		\left|I^{\P}\right|-\left|N(I^{\p})\right| & =\left|I-N_{I}(S)\right|-\left|N(I^{\p})\right|\\
		& =\left|I\right|-\left|N_{I}(S)\right|-\left|N(I)-S\right|\\
		& =\left|I\right|-\left|N_{I}(S)\right|-\left|N(I)\right|+\left|S\right|\\
		& =\left|I\right|-\left|N(I)\right|.
	\end{align*}
	
	\noindent If there exists a set $\emptyset\neq T\s N(I)$ such that
	$\left|N_{I^{\P}}(T)\right|=\left|T\right|$, we repeat the argument.
	Thus we may assume that $\left|N_{I}(S)\right|>\left|S\right|$
	for all $\emptyset\neq S\s N(I)$. Then, by Hall's Theorem,
	there exists a matching of the vertices of $S$ into the vertices of $I-x$
	for every $x\in V(G)$. This matching can be extended to a maximum
	matching in $G$. Therefore $x\in D$ and $I\s D$.
\end{proof} 

By \cref{obscor} and \cref{obsthm}, we obtain the following.

\begin{corollary}
	Let $G$ be an $R$-disjoint graph and $C$ an odd cycle
	of $G$. Then there exists a critical independent set $S$
	of $G[R(C)]$ such that $S\s D(G)$.
\end{corollary}

\begin{lemma}\label{hola}
	Let $G$ be an $R$-disjoint graph and $C$ an odd cycle
	of $G$. Then there exists an independent set $S\in\Omega(G[R(C)])$
	such that $S\s D(G)$.
\end{lemma}

\begin{proof}
	By \cref{obsthm}, $G^{\P}:=G[R(C)]$ is an almost bipartite
	non-König--Egerváry graph, hence $\mu(G)+1\Le\tau(G)$. By
	\cref{lovaszk}, the components of $G[D(G^{\P})]$ consist of isolated vertices
	and the odd cycle $C$. We then define the set $S=S^{\P}\u A(G)$,
	where $S^{\P}$ is a set of $\frac{\left|C\right|+1}{2}$ vertices
	of $C$ covering $C$. Then $S$ is a cover in $G^{\P}$.
	By \cref{ge}
	\[
	\mu(G)=\left|A\right|+\frac{\left|C\right|-1}{2}=\left|S\right|-1.
	\]
	\noindent Therefore, $S$ is a minimum cover of $G^{\P}$,
	and consequently $R(C)-S$ is a maximum independent set of $G^{\P}$
	contained in $D(G)$. 
\end{proof}

As a corollary of \cref{hola} and \cref{obsthm}, we have the additivity
of $\alpha$ in $R$-disjoint graphs, since we can construct an
independent set in $G$ by choosing in each $R(C)$ an independent set
in $D(G)$, and an arbitrary one in $B(G)$.

\begin{theorem}\label{alpha+}
	Let $G$ be an $R$-disjoint graph. Then
	\[
	\alpha\left(G\right)=\alpha\left(B(G)\right)+\underset{C}{\sm}\alpha\left(G[R(C)]\right),
	\]
	\noindent where the sum is over all odd cycles $C$ of
	the graph $G$.
\end{theorem} 

\begin{theorem}\label{inder}
	Let $G$ be an $R$-disjoint graph. Then
	\begin{align*}
		\textnormal{core}(G) & =\U_{C}\textnormal{core}(G[R(C)\u B(G)])\\
		& =\U_{C}\textnormal{ker}(G[R(C)\u B(G)]).
	\end{align*}
\end{theorem}

\begin{proof}
	By \cref{inder}, a maximum independent set of
	$G$ restricted to any set of the flower partition of $G$
	is a maximum independent set in the corresponding induced graph.
	Hence, it is immediate that 
	\[
	\U_{C}\textnormal{core}(G[R(C)\u B(G)])\s\textnormal{core}(G).
	\]
	\noindent Let $x\in\textnormal{core}(G)$, and without loss of generality
	suppose that $x\in R(C)\u B(G)$, for some odd cycle $C$ of
	$G$. Then $x\in\I_{S\in\Omega(G)}S$. Let $S\in\Omega(G[R(C)\u B(G)])$
	and $T(S)\in\Omega(G)$ be a set such that $T(S)\ii(R(C)\u B(G))=S$
	and $T(S)\ii R(C^{\P})\S D(G)$ for every odd cycle $C^{\P}$ distinct
	from $C$, which is possible by \cref{hola} and \cref{tas2}. Then,
	since $x\in R(C)\u B(G)$, we have
	\begin{align*}
		x & \in\left(\I_{S\in\Omega(G[R(C)\u B(G)])}T(S)\right)\ii\left(R(C)\u B(G)\right)\\
		& =\left(\I_{S\in\Omega(G[R(C)\u B(G)])}S\right)\\
		& =\textnormal{core}(G[R(C)\u B(G)]).
	\end{align*}
	\noindent As desired. The second equality follows immediately
	from \cref{obsthm} and \cref{levit}.
\end{proof}

\begin{lemma}
	Let $G$ be an $R$-disjoint graph. Let $H_{1},\dots,H_{k}$
	be the graphs corresponding to the sets of the flower partition. Let $I$ be a critical
	independent set in $G$, and let $I_{i}$ be a critical independent
	set in $H_{i}$, such that $I,I_{i}\s D(G)$ for each $i=1,\dots,k$.
	Then
	\begin{itemize}
		\item $\U_{i}I_{i}$ is a critical independent set in $G$, and 
		\item $I\ii V(H_{i})$ is a critical independent set in $H_{i}$,
		for $i=1,\dots,k$. 
	\end{itemize}
\end{lemma}

\begin{proof}
	Since $I\s D(G)$, by \cref{tas2} 
	\begin{align*}
		\left|I\right|-\left|N(I)\right| & =\sm_{i}\left|I\ii V(H_{i})\right|-\left|N_{H_{i}}(I\ii V(H_{i}))\right|\\
		& \le\sm_{i}\left|I_{i}\right|-\left|N_{H_{i}}(I_{i})\right|\\
		& =\left|\U_{i}I_{i}\right|-\left|N\left(\U_{i}I_{i}\right)\right|,
	\end{align*}
	\noindent hence $\U_{i}I_{i}$ is a critical independent set
	in $G$. On the other hand, if for some $i$, $I\ii V(H_{i})$
	is not critical, then 
	\begin{align*}
		\left|I\right|-\left|N(I)\right| & =\sm_{i}\left|I\ii V(H_{i})\right|-\left|N_{H_{i}}(I\ii V(H_{i}))\right|\\
		& <\left(\sm_{j\neq i}\left|I\ii V(H_{j})\right|-\left|N_{H_{j}}(I\ii V(H_{j}))\right|\right)+\left|I_{i}\right|-\left|N_{H_{i}}(I_{i})\right|\\
		& =\left|\left(\U_{j\neq i}I_{j}\right)\u I_{i}\right|-\left|N\left(\left(\U_{j\neq i}I_{j}\right)\u I_{i}\right)\right|.
	\end{align*}
	\noindent This is absurd, since $I$ is critical in $G$. 
\end{proof}

As an immediate consequence, we have the additivity of the maximum
critical difference $d(G)=\max\{\left|X\right|-\left|N(X)\right|:X\s V(G)\}$.

\begin{corollary}
	Let $G$ be an $R$-disjoint graph. Then 
	\[
	d\left(G\right)=d\left(B(G)\right)+\underset{C}{\sm}d\left(G[R(C)]\right),
	\]
	\noindent where the sum is over all odd cycles $C$ of
	the graph $G$.
\end{corollary} 

\begin{theorem}\label{inder2}
	Let $G$ be an $R$-disjoint graph. Then
	\begin{align*}
		\textnormal{ker}(G) & =\textnormal{ker}(G[B(G)])\u\U_{C}\textnormal{ker}(G[R(C)])=\U_{C}\textnormal{ker}(G[R(C)\u B(G)]).
	\end{align*}
\end{theorem}

\begin{proof}
	Let $X=\{S\s D(G):S\text{ is a critical independent set of \ensuremath{G}}\}$.
	By the proof of \cref{indD}, every critical independent set contains
	a critical independent set in $D(G)$. Therefore,
	\begin{align*}
		\textnormal{ker}(G) & =\I_{S\in X}S\\
		& =\I_{S\in X,S\s D(G)}S\\
		& =\U_{i}\I_{S\in X,S\s D(G)}S\ii V(H_{i})\\
		& =\U_{i}\textnormal{ker}(H_{i})\\
		& =\U_{C}\textnormal{ker}(G[R(C)\u B(G)]),
	\end{align*}
	\noindent where $H_{1},\dots,H_{k}$ are the graphs corresponding
	to the sets of the flower partition.
\end{proof} 

Hence, by \cref{inder} and \cref{inder2}, we obtain the following result.

\begin{theorem}\label{main1}
	If $G$ is an $R$-disjoint graph, then 
	\[
	\textnormal{ker}(G)=\textnormal{core}(G).
	\]
\end{theorem}

\begin{theorem}\label{main2}
	If $G$ is an $R$-disjoint graph, then 
	\[
	\textnormal{corona}(G)\u N(\textnormal{core}(G))=V(G).
	\]
\end{theorem}

\begin{proof}
	Let $x\in V(G)$ and, without loss of generality, suppose
	that $x\in R(C)\u B(G)$ for some odd cycle $C$. Then, by
	\cref{inder} and \cref{inder2}, $x\in\textnormal{corona}(G^{\P})\u N_{G^{\P}}(\textnormal{core}(G))$,
	where $G^{\p}=G[R(C)\u B(G)]$. If $x\in\textnormal{corona}(G^{\P})$, then
	there exists $S\in\Omega(G^{\P})$ such that $x\in S$. Then $x\in S^{\P}\in\Omega(G)$,
	where $S^{\P}\ii\left(R(C)\u B(G)\right)=S$ and $S^{\P}-\left(R(C)\u B(G)\right)\s D(G)$,
	this set exists by \cref{tas2}. On the other hand, if $x\in N_{G^{\P}}(\textnormal{core}(G^{\P}))$,
	then by \cref{levit} $x\in N_{G^{\P}}(\textnormal{ker}(G^{\P}))$, but
	by \cref{inder2} $\textnormal{ker}(G^{\P})\s\textnormal{ker}(G)$, hence $x\in N_{G^{\P}}(\textnormal{ker}(G^{\P}))\s N_{G}(\textnormal{ker}(G))$.
\end{proof}

\begin{theorem}\label{main3}
	If $G$ is an $R$-disjoint graph with exactly
	$k$ disjoint odd cycles, then 
	\[
	\left|\textnormal{corona}(G)\right|+\left|\textnormal{core}(G)\right|=2\alpha(G)+k.
	\]
\end{theorem}

\begin{proof}
	Let $B(G),R_{1},\dots,R_{k}$ be the flower partition of
	$G$. By \cref{inder}, \cref{inder2}, and \cref{tas2}, note that 
	\begin{align*}
		\left|\textnormal{corona}(G)\right| & =\left|\textnormal{corona}(G[B(G)\u R_{1}])\right|+\sm_{i=2}^{k}\left|\textnormal{corona}(G[R_{i}])\right|,\\
		\left|\textnormal{core}(G)\right|=\left|\textnormal{ker}(G)\right| & =\left|\textnormal{ker}(G[B(G)\u R_{1}])\right|+\sm_{i=2}^{k}\left|\textnormal{ker}(G[R_{i}])\right|.
	\end{align*}
	\noindent Therefore, by \cref{levit} and \cref{alpha+}
	\begin{align*}
		\left|\textnormal{corona}(G)\right|+\left|\textnormal{core}(G)\right| & =\left|\textnormal{corona}(G[B(G)\u R_{1}])\right|+\left|\textnormal{ker}(G[B(G)\u R_{1}])\right|+\\
		& +\sm_{i=2}^{k}\left(\left|\textnormal{corona}(G[R_{i}])\right|+\left|\textnormal{ker}(G[R_{i}])\right|\right)\\
		& =2\alpha(G[B(G)\u R_{1}])+1+\sm_{i=2}^{k}\left(2\alpha(G[R_{i}])+1\right)\\
		& =2\alpha(G)+k.
	\end{align*}
	\noindent As desired.
\end{proof}

Hence, from \cref{main1}, \cref{main2}, and \cref{main3}, we recover as a particular case
\cref{levit}.

\section{Consequences and open problems}\label{scon}
In \cite{levit2025almost} the following fact is proved.

\begin{proposition}\label{procj}
	If $G$ is an almost bipartite non-König--Egerváry
	graph, then every maximum matching in $G$ contains at least one edge
	belonging to its unique odd cycle.
\end{proposition}

It is also shown that \cref{procj} is not true for almost
bipartite König--Egerváry graphs. Therefore, the following
conjecture is proposed:
\begin{conjecture}
	[\cite{levit2025almost}] If $G$ is an almost bipartite
	non-König--Egerváry graph, then every maximum matching in $G$ contains
	$\left\lfloor \frac{\left|V(C)\right|}{2}\right\rfloor $ edges belonging
	to its unique odd cycle $C$.
\end{conjecture}

Here we prove a similar result for $R$-disjoint graphs.

\begin{proposition}\label{kkk}
	If $G$ is an $R$-disjoint graph, then
	every maximum matching $M$ of $G$ contains $\left\lfloor \frac{\left|V(C)\right|}{2}\right\rfloor $
	edges from each odd cycle $C$ of $G$.
\end{proposition} 

The proof of \cref{kkk} is essentially a direct application
of \cref{safe}, \cref{obsthm}, and \cref{tas2}.

\begin{proof}
	Let $M$ be a maximum matching of $G$. By \cref{tas2},
	$M^{\p}=M\ii E(G[R(C)])$ is a maximum matching of $G^{\P}=G[R(C)]$.
	But by \cref{obsthm}, $G^{\p}$ is an almost bipartite non-König--Egerváry
	graph. Since $G^{\P}$ has a unique cycle, it cannot have
	an $M^{\P}$-posy. Then, by \cref{safe}, there exists an $M^{\P}$-flower
	$F$ with $M^{\P}$-blossom $C$ in $G^{\P}$. Therefore, $C$ has
	$\left\lfloor \frac{\left|V(C)\right|}{2}\right\rfloor $ edges
	of $M^{\P}$, which are in particular edges of $M$.
\end{proof} 

To conclude the paper, we present some conjectures and open problems
related to this work.

\begin{conjecture}
	The determinant of the adjacency matrix $A(G)$
	of an $R$-disjoint graph admits a factorization (Schur-type functions,
	see \cite{h2,s3}) in terms of the flower decomposition $B(G),R_{1},\dots,R_{k}$, that is,
	\[
	\det A(G)=\det A(G[B(G)])\prod_{i=1}^{k}\det A(G[R_{i}]).
	\]
\end{conjecture}

\begin{problem}
	Characterize graphs satisfying $\left|\textnormal{corona}(G)\right|+\left|\textnormal{core}(G)\right|=2\alpha(G)+k.$ 
\end{problem}

\begin{problem}
	[\cite{levit2025almost}] Characterize graphs satisfying
	$\textnormal{ker}(G)=\textnormal{core}(G).$ 
\end{problem}

\begin{problem}
	[\cite{levit2025almost}] Characterize graphs satisfying
	$\textnormal{corona}(G)\u N(\textnormal{core}(G))=V(G).$ 
\end{problem}

\begin{problem}
	The null space of $A(G)$ of an $R$-disjoint graph
	decomposes orthogonally according to the flower decomposition.
\end{problem}

\begin{problem}
	Characterize $R$-disjoint graphs.
\end{problem}

\section*{Acknowledgments}

This work was partially supported by Universidad Nacional de San Luis (Argentina), PROICO 03-0723, MATH AmSud, grant 22-MATH-02, Agencia I+D+i (Argentina), grants PICT-2020-Serie A-00549 and PICT-2021-CAT-II-00105, CONICET (Argentina) grant PIP 11220220100068CO.

\section*{Declaration of generative AI and AI-assisted technologies in the writing process}
During the preparation of this work the authors used ChatGPT-3.5 in order to improve the grammar of several paragraphs of the text. After using this service, the authors reviewed and edited the content as needed and take full responsibility for the content of the publication.

\section*{Data availability}

Data sharing not applicable to this article as no datasets were generated or analyzed during the current study.

\section*{Declarations}

\noindent\textbf{Conflict of interest} \ The authors declare that they have no conflict of interest.

\bibliographystyle{apalike}

\bibliography{TAGcitasV2025}

\end{document}